\documentclass[11pt]{amsart}
\usepackage{amsmath, amssymb}
\usepackage{amsfonts}
\usepackage{mathrsfs}
\usepackage[arrow,matrix,curve,cmtip,ps]{xy}

\usepackage{amsthm}

\allowdisplaybreaks

\newtheorem{theorem}{Theorem}[section]

\newtheorem{proposition}[theorem]{Proposition}
\newtheorem{corollary}[theorem]{Corollary}

\newtheorem*{theorem*}{Theorem}
\theoremstyle{remark}
\newtheorem{remark}[theorem]{Remark}
\newtheorem{definition}[theorem]{Definition}
\newtheorem{example}[theorem]{Example}

\newcommand{\path}{\operatorname{Path}}

\numberwithin{equation}{section}

\newcommand{\N}{\mathbb{N}}

\newcommand{\prim}{\operatorname{Prim}}
\newcommand{\spec}{\operatorname{Spec}}

\newcommand{\MT}{\operatorname{Max}}
\newcommand{\FR}{\operatorname{FR}}
\newcommand{\CL}{\operatorname{Clust}}

\newcommand{\gae}{\lower 2pt \hbox{$\, \buildrel {\scriptstyle >}\over {\scriptstyle
\sim}\,$}}

\newcommand{\lae}{\lower 2pt \hbox{$\, \buildrel {\scriptstyle <}\over {\scriptstyle
\sim}\,$}}

\newcommand{\MU}[1]{
\setbox0\hbox{$#1$}
\setbox1\hbox{$W$}
\ifdim\wd0>\wd1 #1^{\sim} \else \widetilde{#1} \fi
}

\begin{document}
\title[Prime spectrum and primitive ideal space of a graph $C^*$-algebra]{The prime spectrum and primitive ideal space of a graph $\boldsymbol{C^*}$-algebra}

\author{Gene Abrams}

\author{Mark Tomforde}

\address{Department of Mathematics \\ University of Colorado \\ Colorado Springs, CO 80918 \\USA}
\email{abrams@math.uccs.edu}

\address{Department of Mathematics \\ University of Houston \\ Houston, TX 77204-3008 \\USA}
\email{tomforde@math.uh.edu}

\thanks{This work was supported by Collaboration Grants from the Simons Foundation to each author (Simons Foundation Grant \#20894 to Gene Abrams and Simons Foundation Grant \#210035 to Mark Tomforde)}

\date{\today}



\begin{abstract}

We describe primitive and prime ideals in the $C^*$-algebra  $C^*(E)$ of a graph $E$  satisfying Condition~(K), together with the topologies on each of these spaces.  In particular, we find that primitive ideals correspond to the set of maximal tails disjoint union the set of finite-return vertices, and that prime ideals correspond to the set of clusters of maximal tails disjoint union the set of finite-return vertices.  
\end{abstract}

\maketitle

%%%%%%%%%%%%%%%%%%%%%%%%%%%%%%%%%%%%%%%%%%%%%%%%%%%%%%
\section{Introduction}
%%%%%%%%%%%%%%%%%%%%%%%%%%%%%%%%%%%%%%%%%%%%%%%%%%%%%%

\bigskip

It is well known that any primitive $C^*$-algebra must be a prime $C^*$-algebra, and a partial converse was established by Dixmier in the late 1950's when he showed that every separable prime $C^*$-algebra is primitive (see \cite[Corollaire~1]{DixmierPaper}  or \cite[Theorem~A.49]{RW} for a proof).  Weaver proved in \cite{W} that there exist nonseparable $C^*$-algebras that are prime but not primitive, and in \cite{Kat1, Kat3} Katsura described a class of nonseparable topological graph $C^*$-algebras that are prime but not primitive.

In \cite{AT}, the authors identify the necessary and sufficient graph-theoretic conditions on $E$ for the graph $C^*$-algebra $C^*(E)$ to be prime, and to be primitive.  As one consequence, these descriptions provide a wealth of examples of $C^*$-algebras that are prime but not primitive, thereby adding clarification to the long search and eventual discovery of  such algebras in \cite{W}.

In the current article, we use the descriptions of the prime and primitive  graph $C^*$-algebras to characterize the prime and primitive ideals in the graph $C^*$-algebras arising from graphs that satisfy Condition~(K) (Theorem~\ref{prime-primitive-ideal-description-theorem}).  In addition, for a graph $E$ satisfying Condition~(K) we are able to describe the primitive ideal space $\prim C^*(E)$ and the prime spectrum $\spec C^*(E)$, with their respective topologies, in terms of specified subsets of the vertex set of $E$ (Theorem~\ref{SpecAandPrimAhomeomorphisms}).  

In the countable setting, it is well known that primitive ideals are related to sets of vertices called ``maximal tails".  For a countable graph, a maximal tail is  defined to be a set of vertices satisfying three axioms, known as (MT1), (MT2), and (MT3), and it is proven that a set satisfies these three axioms if and only if its complement is equal to the set of all vertices that can reach a ``tail" (i.e. an infinite path or a finite path ending at a singular vertex) --- indeed, it is this fact that motivates the name ``maximal tail".  We show that for (not necessarily countable) graphs it is useful to add a condition (MT4) to the definition of a ``maximal tail".  The new condition (MT4) is automatically satisfied in the countable setting, but does not hold for arbitrary graphs. We show that a set of vertices satisfies (MT1), (MT2), (MT3), and (MT4) if and only if its complement is equal to the set of all vertices that can reach a ``tail", and in analogy with the countable graph case, we call such a set a \emph{maximal tail}.  Additionally, we call a set that is only required to satisfy conditions (MT1), (MT2), and (MT3) a \emph{cluster of maximal tails}.  In a countable graph any cluster of maximal tails is automatically a maximal tail, but this is not true in general.  We then show that for a graph $E$ satisfying Condition~(K), the primitive ideals in $C^*(E)$ correspond to the (disjoint) union of the set of maximal tails with the set of ``finite-return vertices", and that the prime ideals correspond to the (disjoint) union of the  set of clusters of maximal tails with the set of ``finite-return vertices".  In particular, ideals coming from ``finite-return vertices" are always primitive, and the only prime,  non-primitive ideals come from clusters of maximal tails that are not maximal tails.

%%%%%%%%%%%%%%%%%%%%%%%%%%%%%%%%%%%%%%%%%%%%%%%%%%%%%%
\section{Preliminaries}
%%%%%%%%%%%%%%%%%%%%%%%%%%%%%%%%%%%%%%%%%%%%%%%%%%%%%%

We now briefly review the relevant terminology, notation, and constructions.  For a more complete description of these notions, see \cite{AT}.  While it would perhaps be appropriate to first give all of the germane graph-theoretic terminology and then subsequently discuss all of the relevant $C^*$-algebra results, we will instead  intersperse the graph-theoretic terms throughout  our analysis of graph $C^*$-algebras  as needed.

Throughout this article a set will be called \emph{countable} if it is either finite or countably infinite.  A \emph{graph} $(E^0, E^1, r, s)$ consists of a set $E^0$ of vertices, a set $E^1$ of edges, and maps $r : E^1 \to E^0$ and $s : E^1 \to E^0$ identifying the range and source of each edge.    We note that no restriction is placed on the cardinality of $E^0$, or of $E^1$, or of $s^{-1}(v)$ for any vertex $v\in E^0$.  

 A vertex $v
\in E^0$ is a \emph{sink} if $s^{-1}(v) = \emptyset$, while $v$ is an \emph{infinite emitter} if
$|s^{-1}(v)| = \infty$.  A \emph{singular vertex} is a vertex that
is either a sink or an infinite emitter, and we denote the set of
singular vertices by $E^0_\textnormal{sing}$.  We define 
$E^0_\textnormal{reg} := E^0 \setminus E^0_\textnormal{sing}$, and
refer to the elements of $E^0_\textnormal{reg}$ as \emph{regular
vertices}; i.e., a vertex $v \in E^0$ is a regular vertex if and
only if $0 < |s^{-1}(v)| < \infty$.  

For a graph $E$ we construct a $C^*$-algebra $C^*(E)$ as follows.

\begin{definition} \label{graph-C*-def}
If $E$ is a graph, the \emph{graph $C^*$-algebra} $C^*(E)$ is the universal
$C^*$-algebra generated by mutually orthogonal projections $\{ p_v
: v \in E^0 \}$ and partial isometries with mutually orthogonal
ranges $\{ s_e : e \in E^1 \}$ satisfying
\begin{enumerate}
\item $s_e^* s_e = p_{r(e)}$ \quad  for all $e \in E^1$
\item $s_es_e^* \leq p_{s(e)}$ \quad for all $e \in E^1$
\item $p_v = \sum_{\{ e \in E^1 : s(e) = v \}} s_es_e^* $ \quad for all $v \in E^0_\textnormal{reg}$.
\end{enumerate}
\end{definition}

\noindent We call Conditions (1)--(3) in Definition~\ref{graph-C*-def} the \emph{Cuntz-Krieger relations}.  

A graph is \emph{row-finite}
if it has no infinite emitters.  A graph is \emph{finite} if both $E^0$ and $E^1$ are finite sets.  A graph is \emph{countable} if both $E^0$ and $E^1$ are countable sets.  It is straightforward to show that $E$ is countable if and only if $C^*(E)$ is a separable $C^*$-algebra.

A  \emph{path} $\alpha$ in a graph $E$ is a sequence $\alpha := e_1 e_2
\ldots e_n$ of edges with $r(e_i) = s(e_{i+1})$ for $1 \leq i \leq
n-1$.  We say the path $\alpha$ has \emph{length} $| \alpha| :=n$,
and we let $E^n$ denote the set of paths of length $n$.  We
consider the vertices in $E^0$ to be paths of length zero.  We
define  $\path (E) := \bigcup_{n\geq 0}  E^n$, and we extend the maps $r$ and $s$ to $\path (E)$
as follows: for $\alpha = e_1 e_2 \ldots e_n \in E^n$ with $n\geq
1$, we set $r(\alpha) = r(e_n)$ and $s(\alpha) = s(e_1)$; for
$\alpha = v \in E^0$, we set $r(v) = v = s(v)$.  Also, for $\alpha = e_1e_2\cdots e_n\in \path (E)$,  we let $\alpha^0$ denote the set of vertices that appear in $\alpha$; that is, $$\alpha^0  = \{ s(e_1), r(e_1), \ldots, r(e_n) \}.$$  If $v, w \in E^0$ we write $v \geq w$ to mean that there exists a path $\alpha \in \path (E)$ with $s(\alpha) = v$ and $r(\alpha) = w$. If $A, B \subseteq E^0$, we write $A \geq B$ if for every $v \in A$ there exists $w \in B$ such that $v \geq w$.

\begin{definition}\label{downdirectdef}
A graph $E$ is called  \emph{downward directed} if for all $u, v \in E^0$, there exists $w \in E^0$ such that $u \geq w$ and $v \geq w$.
\end{definition}

For a subset $S\subseteq E^0$, we define $$U(S) := \{ v \in E^0 : v \geq s \ \mbox{ for some } s\in S \};$$ that is,  $U(S)$ is the set of vertices $v$ for which there exists a path from $v$ to $s$, where $s$ is some element of $S$.   In case the set $S$ is a single vertex $w$, we denote $U(\{w\})$ simply by $U(w)$, and we see that $U(S) = \bigcup_{s\in S}U(s)$.   We note that $S\subseteq U(S)$ for any $S$, as $v\geq v$ for each $v\in E^0$ by definition.

\begin{definition}
Let $E = (E^0, E^1, r, s)$ be a graph.   We say  $E$ satisfies the {\it Countable Separation Property}  if there exists a countable set $S\subseteq E^0$ for which $E^0 = U(S)$.
\end{definition}

Trivially any countable graph satisfies the Countable Separation Property, using $S = E^0$.  

\smallskip

A \emph{cycle} $\alpha$ in a graph $E$  is a path $\alpha=e_1 e_2 \ldots e_n$ with length $|\alpha| \geq 1$ and $r(\alpha) = s(\alpha)$.  If $\alpha = e_1e_2 \ldots e_n$ is a cycle, an \emph{exit} for $\alpha$ is an edge $f \in E^1$ such that $s(f) = s(e_i)$ and $f \neq e_i$ for some $i$.  We say that a graph satisfies \emph{Condition~(L)} if every cycle in the graph has an exit.
A \emph{simple cycle} is a cycle $\alpha=e_1 e_2 \ldots e_n$ with $r(e_i) \neq s(e_1)$ for all $1 \leq i \leq n-1$.  We say that a graph satisfies \emph{Condition~(K)} if no vertex in the graph is the source of exactly one simple cycle.  (In other words, a graph satisfies Condition~(K) if and only if every vertex in the graph is the source of no simple cycles or the source of at least two simple cycles.)

We will use the following graph as a running example throughout the paper to illustrate many of our ideas.

\begin{example}\label{EsubTsubXexample}
Let $X$ be any nonempty set, and let $\mathcal{F}(X)$ denote the collection of finite nonempty subsets of $X$.   We define the graph $E_{A}(X)$ as follows.
\begin{align*}
E_{A}(X)^0 \ &:= \ \{v_A : A \in \mathcal{F}(X)\}, \\
E_{A}(X)^1 &:= \{e_{A,A'} : A,A'\in \mathcal{F}(X) \ \mbox{and} \ A\subsetneqq A'\},
\end{align*}
$s(e_{A,A'})=v_A$ for each $e_{A,A'} \in E_{A}(X)^1$, and $r(e_{A,A'})=v_{A'}$ for each $e_{A,A'} \in E_{A}(X)^1$.

There are some immediate observations about the graph $E_A(X)$ that will be very useful throughout this article:
For any vertex $v_A \in E_A(X)$ the set $A$ is finite, and $U(v_A) = \{v_B : B \subseteq A \text{ and } B \neq \emptyset \}$ is a finite set.  In addition, we see that $E_A(X)$ contains no cycles.  Furthermore, if $X$ is an infinite set, then every vertex of $E_A(X)$ is an infinite emitter. 
\end{example}

Two $C^*$-algebraic concepts that play the central roles in this article are \emph{prime} and \emph{primitive}.  We define both now, and note that they are both related to the ideals in a $C^*$-algebra.  For us, the term \emph{ideal} shall always mean a closed, two-sided ideal of a $C^*$-algebra.

\begin{definition} \label{prime-C*-def}
A $C^*$-algebra $A$ is {\it prime} if whenever $I$ and $J$ are closed two-sided ideals of $A$ and $IJ = \{ 0 \}$, then either $I = \{ 0 \}$ or $J = \{ 0 \}$. 
\end{definition}

We note that $IJ = I \cap J $ for any two-sided ideals of a $C^*$-algebra $A$, so that $A$ is prime if whenever $I$ and $J$ are closed two-sided ideals of $A$ and $I\cap J = \{ 0 \}$, then either $I = \{ 0 \}$ or $J = \{ 0 \}$.
  
\begin{definition}  \label{primitive-C*-def}
A $C^*$-algebra $A$ is \emph{primitive} if there exists a faithful irreducible $*$-representation $\pi : A \to B(\mathcal{H})$ for some Hilbert space $\mathcal{H}$. 
\end{definition}

We note that the primitivity of a $C^*$-algebra $A$ as presented in Definition~\ref{primitive-C*-def} coincides with the left (or right) primitivity of $A$ as a ring; see e.g. \cite[Proposition~2.25]{AT}.     The statements in the following remark are well known.  

\begin{remark}\label{primitiveimpliesprime}   
Any  primitive $C^*$-algebra  is prime (see \cite[Lemma 3.2]{AT} for a proof).  Any prime $C^*$-algebra that is also separable is primitive (see either \cite[Corollaire~1]{DixmierPaper} or \cite[Theorem~A.49]{RW} for a proof).  However, there exist nonseparable $C^*$-algebras that are prime but not primitive (see \cite{W} for the first known example).
\end{remark}

\begin{proposition}[Proposition~3.1 of \cite{AT}] \label{primeprop}  
Let $E$ be any graph.  Then $C^*(E)$ is prime if and only if the following two properties hold:
\begin{itemize}
\item[(i)] $E$ satisfies Condition~(L), and
\item[(ii)] $E$ is downward directed.   
\end{itemize}
\end{proposition}

\noindent The next theorem is the main result of \cite{AT}, and we will rely on it heavily in our discussions.

\begin{theorem}[Theorem~3.8 of \cite{AT}] \label{primitivitytheorem} 
Let $E$ be any graph.  Then $C^*(E)$ is primitive if and only if the following three properties hold:
\begin{itemize}
\item[(i)] $E$ satisfies Condition~(L),
\item[(ii)] $E$ is downward directed, and
\item[(iii)] $E$ satisfies the Countable Separation Property.
\end{itemize}
In other words, by Proposition~\ref{primeprop}, $C^*(E)$ is primitive if and only if $C^*(E)$ is prime and $E$ satisfies the Countable Separation Property. 
\end{theorem}

\begin{example} \label{prime-not-prim-uncountable-ex}
Let $X$ be any nonempty set and let $E_A(X)$ be the graph from Example~\ref{EsubTsubXexample}.  Since $E_A(X)$ has no cycles, $E_A(X)$ satisfies Condition~(L) vacuously.  In addition, $E_A(X)$ is downward directed because if $v_A, v_B \in E_A(X)^0$ and we define $C:= A \cup B$, then $v_C \in E_A(X)^0$ with $v_A \geq v_C$ and $v_B \geq v_C$.  It follows from Proposition~\ref{primeprop} that $C^*(E_A(X))$ is a prime $C^*$-algebra.

In addition, for any $v_A \in E_A(X)^0$ we have $U(v_A) = \{v_B : B \subseteq A \text{ and } B \neq \emptyset \}$ is a finite set.  Hence $\bigcup_{v_A \in E_A(X)^0} U(v_A)$ is countable if and only if $E_A(X)^0$ is countable if and only if $X$ is countable.  It follows that $E_A(X)$ has the Countable Separation Property if and only if $X$ is countable, and by Theorem~\ref{primitivitytheorem} $C^*(E_A(X))$ is primitive if and only if $X$ is countable.  Thus for any uncountable set $X$, the graph $C^*$-algebra $C^*(E_A(X))$ is a nonseparable $C^*$-algebra that is prime but not primitive.  This gives us a wealth of examples of prime but not primitive $C^*$-algebras.
\end{example}

%%%%%%%%%%%%%%%%%%%%%%%%%%%%%%%%%%%%%%%%%%%%%%%%%%%%%%
\section{Prime and primitive ideals of a graph $C^*$-algebra}
%%%%%%%%%%%%%%%%%%%%%%%%%%%%%%%%%%%%%%%%%%%%%%%%%%%%%%

Our preliminary discussion has centered around the notions of primeness and primitivity of a $C^*$-algebra.   These two notions extend in a natural way to the ideals of a $C^*$-algebra, as follows.  

\begin{definition}
Let $A$ be a $C^*$-algebra, and let $I$ be an ideal in $A$.  The ideal $I$ is \emph{primitive} if there exists a faithful irreducible $*$-representation $\pi : A \to B(\mathcal{H})$ for some Hilbert space $\mathcal{H}$, such that $I = \ker \pi$.  The ideal $I$ is \emph{prime} if whenever $J$ and $K$ are ideals of $A$ and $J \cap K \subseteq I$, then either $J \subseteq I$ or $K \subseteq I$.
\end{definition}

Let $A$ be a $C^*$-algebra and let $I$ be an ideal of $A$.  Observe that $I$ is a primitive ideal if and only if $A/I$ is a primitive $C^*$-algebra.  Likewise, $I$ is a prime ideal if and only if $A/I$ is a prime $C^*$-algebra.  In particular,  if $A$ is any $C^*$-algebra, then $A$ is a primitive $C^*$-algebra precisely when $\{ 0 \}$ is a primitive ideal of $A$, and $A$ is a prime $C^*$-algebra precisely when  $\{ 0 \}$ is a prime ideal of $A$.

Our goal is to describe the prime and primitive ideals of a graph $C^*$-algebras $C^*(E)$ in terms of explicit subsets of the vertices of $E$.  As we shall see, subsets called {\it maximal tails} and {\it clusters of maximal tails} will play key roles.

\begin{definition}\label{tailsdefinition}
Let $E$ be a graph.   An \emph{infinite path in} $E$  is a sequence $\lambda = e_1, e_2, \dots$ having $r(e_i) = s(e_{i+1})$ for all $i\geq 1$.   We denote the set of infinite paths in $E$ by $E^\infty$.  (We note that an infinite path is not an element of ${\rm Path}(E)$.)   For $\lambda = e_1, e_2, ... \in E^\infty$ we denote the set $\{s(e_i) : i\in \N\}$ by $\lambda^0$.  The \emph{boundary paths} of $E$ are the elements of the set
$$\partial E := E^\infty \cup \{ \alpha \in \bigcup_{n \geq 0} E^n : r(\alpha) \in E^0_\textnormal{sing} \}.$$
%For any boundary path $\alpha \in \partial E$, we let $$\alpha^0 := \{ s(\alpha_1) \} \cup \bigcup_{i=1}^{|\alpha|} \{ r(\alpha_i) \}$$ and note that $|\alpha|$ may be infinite.  
We emphasize that each element of $E^0_{\rm{sing}}$ (i.e., each sink and each infinite emitter of $E$) is in $\partial E$.   For any $\alpha \in \partial E$ we define the \emph{maximal   tail determined by $\alpha$} to be the set
$$ \ T_\alpha \ :=  \  U(\alpha^0) \ = \ \{ v \in E^0 : v \geq \alpha^0 \}.$$
\noindent
In particular, if $w\in E^0_{\rm{sing}}$ then $T_w = \{ v\in E^0 : v \geq w\}$. 
 \begin{enumerate}
\item We call a subset $T \subseteq E^0$ a \emph{maximal tail} if $T=T_\alpha$ for some $\alpha \in \partial E$.  

\smallskip

\item We call a subset $M \subseteq E^0$ a \emph{union of maximal tails} if there is a set of boundary paths $S \subseteq \partial E$ such that $M = \bigcup_{\alpha \in S} T_\alpha$.  

\smallskip

\item We call a subset $C \subseteq E^0$ a \emph{cluster of maximal tails} if $C$ is a union of maximal tails and $C$ is downward directed (see Definition~\ref{downdirectdef}).
\end{enumerate}
\end{definition}

Let $T_\alpha$ be a maximal tail, and write $\alpha \in \partial E$ as $ e_1e_2 ...$, where $\alpha$ either has finite length (and ends at a singular vertex), or $\alpha \in E^\infty$.  Let $u,v \in T_\alpha$.   Then there exists $i,j \in \N$ for which $u \geq s(e_i)$ and $v\geq s(e_j)$, and if we set $k = {\rm max}(i,j)$, then $u\geq s(e_k)$ and $v\geq s(e_k)$.  Hence $T_\alpha$ is downward directed.   

We see that every maximal tail is a cluster of maximal tails (since every maximal tail is downward directed), and that every cluster of maximal tails is a union of maximal tails (by definition).  However, neither of these implications can be reversed in general:  the set of vertices of the graph $\xymatrix{ \bullet & \bullet \ar[r] \ar[l] & \bullet}$ is a union of maximal tails that is not a cluster of maximal tails, and the following example provides an example of a cluster of maximal tails that is not a maximal tail.

\begin{example}\label{clustermaxtailsnotamaxtailexample}
Let $X$ be an uncountable set, and let $E_A(X)$ be the graph of Example~\ref{EsubTsubXexample}.  Let $C := E_A(X)^0$.  Every vertex in $E_A(X)$ is an infinite emitter, so $C = \bigcup_{v \in E_A(X)^0} T_v$ is a union of maximal tails, and since $E_A(X)$ is downward directed (see Example~\ref{prime-not-prim-uncountable-ex}) $C$ is also a cluster of maximal tails.  However, for any maximal tail $T_\alpha$ in $E_A(X)$, we have $T_\alpha = U(\alpha^0) = U(s(\alpha)) \cup \bigcup_{i=1}^{|\alpha|} U(r(\alpha_i))$.  Since $U(v_A) = \{ B : B \subseteq A \text{ and } B \neq \emptyset \}$ is finite for all $v_A \in E_A(X)^0$, it follows that $T_\alpha$ is a countable set.  Since $X$ is uncountable, $E_A(X)$ has an uncountable number of vertices.  Hence $C$ is uncountable, and $C$ is not a maximal tail.
\end{example}

\noindent Corollary~\ref{countable-clust-max-tail-cor} shows that clusters of maximal tails that are not maximal tails can only occur in uncountable graphs.

\begin{definition}
Let $E$ be a graph and $S \subseteq E^0$. The following properties of $S$ were studied in \cite[\S6]{BPRS} and \cite[\S1]{HS4}.  The terminology ``MT" is meant to stand for ``Maximal Tail".  

$ $

\begin{itemize}
\item[(MT1)] If $v \in E^0$, $w \in S$, and $v \geq w$, then $v \in S$.  
\item[(MT2)] If $v \in S$ and $v \in E^0_\textnormal{reg}$, then there exists $e \in E^1$ with $s(e) = v$ and $r(e) \in S$.  
\item[(MT3)] $S$ is downward directed.  
\end{itemize}

$ $

\noindent In addition to these previously studied properties, we introduce a fourth property that will be useful for us.

\smallskip

\begin{itemize}
\item[(MT4)] $S$ satisfies the Countable Separation Property. 
\end{itemize}

$ $

\noindent In \cite[\S6]{BPRS} and \cite[\S1]{HS4} the authors  consider  only countable graphs, and hence all their graphs trivially satisfy (MT4).  In our investigations, we allow uncountable graphs, and we find that (MT4) is an important property that  is not automatically satisfied and must be analyzed.
\end{definition}

\begin{remark}
Note that if $S \subseteq E^0$ satisfies (MT1), then whenever $A \subseteq E^0$ and $A \geq S$, it follows that $A \subseteq S$.
\end{remark}

We note also that the empty set $\emptyset$ vacuously satisfies each of the four (MT) conditions.

The next proposition shows that the (MT1)--(MT4) properties of Definition~\ref{tailsdefinition} can be used to characterize maximal tails as well as unions and clusters of maximal tails.

\begin{proposition} \label{MT-axioms-tails-prop}
Let $E = (E^0, E^1, r, s)$ be a graph.
\begin{itemize}
\item[(1)] A subset $M \subseteq E^0$ is a union of maximal tails if and only if $M$ satisfies (MT1) and (MT2).
\item[(2)] A subset $C \subseteq E^0$ is a cluster of maximal tails if and only if $C$ satisfies (MT1), (MT2), and (MT3).
\item[(3)] A subset $T \subseteq E^0$ is a maximal tail if and only if $T$ satisfies (MT1), (MT2), (MT3), and (MT4).
\end{itemize}
\end{proposition}

\begin{proof}
For the forward implication in (1), we let $M \subseteq E^0$ and suppose $M$ is a union of maximal tails.  Then there exists a set of boundary paths $S \subseteq \partial E$ such that $M = \bigcup_{\alpha \in S} T_\alpha$.  Let $v \in E^0$ and $w \in M$ with $v \geq w$.  Since $w \in M$, we have $w \in T_\alpha$ for some $\alpha \in S$.  Thus $w \geq \alpha^0$, and since $v \geq w$, it follows that $v \geq \alpha^0$ and $v \in T_\alpha \subseteq M$.  Thus $M$ satisfies (MT1).  In addition, if $v \in M$ and $v \in E^0_\textnormal{reg}$, then the fact that $v \in M$ implies that $v \in T_\alpha$ for some $\alpha \in S$, and hence $v \geq \alpha^0$.  If $v \notin \alpha^0$, then there exists an edge $e \in E^1$ such that $s(e) = v$ and $r(e) \geq \alpha^0$, so that $r(e) \in T_\alpha \subseteq M$.  On the other hand, if $v \in \alpha^0$, then the fact that $v \in E^0_\textnormal{reg}$ and the definition of $\partial E$ implies that there exists an edge $e$ of the path $\alpha$ with $s(e) = v$ and $r(e) \in \alpha^0$, so that $r(e) \in T_\alpha \subseteq M$.  Thus $M$ satisfies (MT2).

For the converse implication in (1), let $M \subseteq E^0$ and suppose that $M$ satisfies (MT1) and (MT2).  Property~(MT2) implies that for each $w \in M$, either $w \in E^0_\textnormal{sing}$ or there exists an edge $e$ with $s(e) = w$ and $r(e) \in M$.  We may continue this process recursively, to produce for each $w \in M$ a boundary path $\alpha_w \in \partial E$ with $s(\alpha_w) = w$.  We let $T_{\alpha_w}$ be the associated maximal tail, and consider $\bigcup_{w \in M} T_{\alpha_w}$.   We shall show that $M = \bigcup_{w \in M} T_{\alpha_w}$.  If $v \in M$, then $s(\alpha_v) = v$, so $v \in T_{\alpha_v} \subseteq  \bigcup_{w \in M} T_{\alpha_w}$.  Thus $M \subseteq \bigcup_{w \in M} T_{\alpha_w}$.  For the reverse inclusion, let $v \in \bigcup_{w \in M} T_{\alpha_w}$, then $v \in T_{\alpha_w}$ for some $w \in M$.  Hence $v \geq \alpha_w^0$ and there exists $x \in \alpha_w^0$ such that $v \geq \alpha_w^0$.   By the way $\alpha_w$ was produced, $\alpha_w^0 \subseteq M$.  Thus there exists $x \in M$ such that $v \geq x$.  Since $x \in M$, it follows from (MT1) that $v \in M$.  Thus $M = \bigcup_{w \in M} T_{\alpha_w}$.   Hence $M$ is a union of maximal tails. 

For (2), suppose that $C \subseteq E^0$ is a cluster of maximal tails.  By definition $C$ is a union of maximal tails that is downward directed.  Thus $C$ satisfies (MT1) and (MT2) by part (1), and since $C$ is downward directed $C$ satisfies (MT3).  Conversely, if $C \subseteq E^0$ and $C$ satisfies (MT1), (MT2), and (MT3), then part (1) implies $C$ is a union of maximal tails, and (MT3) implies $C$ is downward directed, so $C$ is a cluster of maximal tails.

For the forward implication in (3), suppose $T \subseteq E^0$ is a maximal tail.  Then $T = T_\alpha$ for some $\alpha \in E^0$.  Since any maximal tail is a cluster of maximal tails, it follows from part~(2) that $T$ satisfies (MT1), (MT2), and (MT3).  In addition, if we let $X := \alpha^0$, then $X$ is a countable set and we see that $T = T_\alpha = \{ v \in E^0 : v \geq \alpha^0 \} = \bigcup_{x \in X} U(x)$, so that $T$ satisfies (MT4).  

For the converse implication in (3), let $T \subseteq E^0$ and suppose that $T$ satisfies (MT1), (MT2), (MT3), and (MT4).  By property~(MT4), there is a countable set $X \subseteq T$ such that $T = \bigcup_{x \in X} U(x)$.  List the elements of $X$ as $X = \{ x_1, x_2, x_3, \ldots \} = \{ x_k \}_{k \in I}$.  Note that the index set $I$  may be finite or countably infinite.  Choose any vertex $v_1 \in T$, and use (MT3) to choose a vertex $v_2 \in T$ such that $x_1 \geq v_2$ and $v_1 \geq v_2$.  Next, use (MT3) to choose a vertex $v_3 \in T$ such that $x_2 \geq v_3$ and $v_2 \geq v_3$.  In this way, we may recursively construct a (finite or countably infinite) list of vertices $v_1, v_2, v_3, ... \in T$ such that $x_k \geq v_{k+1}$ and $v_k \geq v_{k+1}$ for all $k \in I$.  Since $v_1 \geq v_2 \geq v_3 \geq \ldots$, for each $k$ we may choose a (finite)  path $\lambda_k$ with $s(\lambda_k) = v_k$ and $r(\lambda_k) = v_{k+1}$.  Then $\alpha := \lambda_1 \lambda_2 \lambda_3 \ldots$ is a (finite or infinite) path with $x \geq \alpha^0$ for all $x \in X$.  Since $r(\lambda_k) = v_{k+1} \in T$ for all $k\in I$, it follows from (MT1) that $\alpha^0 \subseteq T$.  If $\alpha$ is an infinite path, then $\alpha \in \partial E$.  If $\alpha$ is a finite path, then either $r(\alpha) \in E^0_\textnormal{sing}$ or $r(\alpha) \in E^0_\textnormal{reg}$.  If $r(\alpha) \in E^0_\textnormal{sing}$, then $\alpha \in \partial E$.  If $r(\alpha) \in E^0_\textnormal{reg}$, we may use (MT2) to produce an edge $e \in E^1$ such that $s(e) = r(\alpha)$ and $r(e) \in T$.  Continuing this process, we may extend $\alpha$ to either an infinite path or a finite path whose range is in $E^0_\textnormal{sing}$.  In either case this extension, which we shall also denote by $\alpha$, is an element of $\partial E$.  Hence in any case we may produce a boundary path $\alpha \in \partial E$ with $\alpha^0 \subseteq T$ and $x \geq \alpha^0$ for all $x \in X$.  Since $\alpha^0 \subseteq T$, (MT1) implies that $T_\alpha \subseteq T$.  For the reverse inclusion, let $v \in T$.  The fact that $T = \bigcup_{x \in X} U(x)$ implies there exists $x \in X$ such that $v \geq x$.  By the construction of $\alpha$, we have that $x \geq \alpha^0$.  Hence $v \geq \alpha^0$, and it follows that $x \in T_\alpha$.  Thus $T = T_\alpha$, and $T$ is a maximal tail.
\end{proof}

\begin{corollary} \label{countable-clust-max-tail-cor}
Let $E = (E^0, E^1, r, s)$ be a graph.  If $M \subseteq E^0$ is a cluster of maximal tails and $M$ is countable, then $M$ is a maximal tail.  In particular, in a countable graph, a subset of vertices is a cluster of maximal tails if and only if that subset is a maximal tail. 
\end{corollary}

\begin{remark}
It follows from Proposition~\ref{MT-axioms-tails-prop} that \emph{maximal tails} are precisely those subsets of vertices satisfying properties (MT1)--(MT4), and \emph{clusters of maximal tails} are precisely those subsets of vertices satisfying properties (MT1)--(MT3).  Since (MT4) is automatically satisfied for any subset of a countable graph, the notions of maximal tails and clusters of maximal tails coincide in countable graphs.  In the literature predating this article (e.g., \cite{BPRS}, \cite{BHRS}, and \cite{HS4}), the definition of a ``maximal tail" was  taken to be a subset of vertices satisfying (MT1)--(MT3).  Since most of the prior graph $C^*$-algebra literature has (often implicitly) worked under the assumption that the graphs considered are countable, these earlier definitions of a ``maximal tail" agree with ours in the setting of countable graphs.  However, for our purposes in working with possibly uncountable graphs, it is important that we add property (MT4) to the definition of a ``maximal tail", and that we distinguish between the notions of ``maximal tail" and ``cluster of maximal tails".
\end{remark}

\begin{definition} \label{graph-ideals-quotients-def}
If $E = (E^0, E^1, r, s)$ is a graph, a subset $H \subseteq E^0$ is \emph{hereditary} if $e \in E^1$ and $s(e) \in H$ implies $r(e) \in H$.  A hereditary subset $H$ is called \emph{saturated} if $v \in E^0_\textnormal{reg}$ and $r (s^{-1}(v)) \subseteq H$ implies that $v \in H$.  For any saturated set $H$, we let $$B_H := \{ v \in E^0_\textnormal{sing} : 0 < | r(s^{-1}(v)) \cap (E^0 \setminus H) | < \infty \}$$
denote the set of \emph{breaking vertices} of $H$.  An \emph{admissible pair} $(H,S)$ is a pair consisting of a saturated hereditary set $H$ and a subset $S \subseteq B_H$ of breaking vertices for $H$. 

 For an admissible pair $(H,S)$ in $E$, we define the graph $E \setminus (H,S)$ as follows:  $$(E \setminus (H,S))^0 := (E^0 \setminus H) \cup \{v' : v \in B_H \setminus S \},$$ edge set 
$$(E \setminus (H,S))^1 := s^{-1} (H) \cup \{e' : e \in E^1 \text{ and } r(e) \in B_H \setminus S \},$$ and range and source maps extended from $r|_{s^{-1} (H)}$ and $s|_{s^{-1} (H)}$ by defining $r(e') := r(e)'$ and $s(e') := s(e)$. We note that for $e' \in E \setminus (H,S)$, the vertex $r(e')$ is a sink. 
\end{definition}

 If $(H,S)$ is an admissible pair in $E$ and $v\in S$, we define the \emph{gap projection} corresponding to $v$ to be the following element of $C^*(E)$:
 $$p_v^H := p_v - \sum_{\{ e \in E^1 : s(e) = v \text{ and } r(e) \notin H \} } s_es_e^*.$$  Note that the sum appearing is finite and nonzero. For any graph $E = (E^0, E^1, r, s)$ and any admissible pair $(H,S)$ of $E$, we define $I_{(H,S)}$ to be the closed two-sided ideal in $C^*(E)$ generated by the set 
 $$\{ p_v : v \in H \} \cup  \{p_v^H: v\in S\}.$$

It is shown in \cite[Corollary~3.5]{BHRS} that for any graph $E$ and any admissible pair $(E,H)$ of $E$, one has $$C^*(E) / I_{(H,S)} \cong C^*(E \setminus (H,S)).$$  In addition, it is shown in \cite[Theorem~3.6]{BHRS} that for any graph $E$ the map $(H,S) \mapsto I_{(H,S)}$ is a bijection from the collection of admissible pairs of $E$ onto the collection of gauge-invariant ideals of $C^*(E)$.  Also, it is shown in \cite[Theorem~2.1.19]{Tom9} that $E$ satisfies Condition~(K) if and only if all ideals of $C^*(E)$ are gauge invariant.  (Although the results of \cite[Corollary~3.5]{BHRS}, \cite[Theorem~3.6]{BHRS}, and \cite[Theorem~2.1.19]{Tom9} are stated for countable graphs, the results still hold and the proofs go through verbatim when the graphs are not necessarily countable.)  These results imply that when $E$ satisfies Condition~(K) the map $(H,S) \mapsto I_{(H,S)}$ is a bijection from the collection of admissible pairs of $E$ onto the collection of ideals of $C^*(E)$.  

\begin{remark}\label{KimpliesquotientshaveL}
It is well known (and a fairly straightforward exercise) to show that a graph $E$ satisfies Condition~(K) if and only if for each admissible pair $(E,H)$ the graph $E \setminus (H,S)$ satisfies Condition~(L).
\end{remark}

The following is one of the main results of this paper and allows us to describe prime and primitive ideals in the $C^*$-algebra of a graph satisfying Condition~(K).

\begin{theorem} \label{prime-primitive-ideal-description-theorem}
Let $E = (E^0, E^1, r, s)$ be a graph that satisfies Condition~(K), and let $I \triangleleft C^*(E)$ be a closed two-sided ideal of $C^*(E)$.  Then the following statements hold. 
\begin{itemize}
\item[(1)] $I = I_{(H,S)}$ where $E^0 \setminus H$ is a union of maximal tails and $S \subseteq B_H$.
\item[(2)] $I$ is a primitive ideal if and only if one of the following two situations occurs:
\begin{itemize}
\item[(a)] $I = I_{(H, S)}$ with $E^0 \setminus H$ a maximal tail and $S = B_H$.
\item[(b)] $I = I_{(H, S)}$ with $E^0 \setminus H = T_{v_0}$ for some $v_0 \in B_H$ and $S = B_H \setminus \{ v_0 \}$.
\end{itemize}
\item[(3)] $I$ is a prime ideal if and only if one of the following two situations occurs:
\begin{itemize}
\item[(a)] $I = I_{(H, S)}$ with $E^0 \setminus H$ a cluster of maximal tails and $S = B_H$.
\item[(b)] $I$ is a primitive ideal as described in case 2(b); namely, that  $I = I_{(H, S)}$ with $E^0 \setminus H = T_{v_0}$ for some $v_0 \in B_H$ and $S = B_H \setminus \{ v_0 \}$.
\end{itemize}
\end{itemize} 
\end{theorem}

\begin{proof}
Throughout the proof we use the notation and terminology of Definition~\ref{graph-ideals-quotients-def}.

For part~(1), since $E$ satisfies Condition~(K), all ideals of $C^*(E)$ are gauge-invariant and $I = I_{(H,S)}$ for some saturated hereditary subset $H \subseteq E^0$ and some set of breaking vertices $S \subseteq B_H$.  Since $H$ is hereditary, $E^0 \setminus H$ satisfies (MT1).  Since $H$ is saturated, $E^0 \setminus H$ satisfies (MT2).  It follows from Proposition~\ref{MT-axioms-tails-prop} that $E^0 \setminus H$ is a union of maximal tails.

For part~(2), first suppose that $I$ is a primitive ideal.  From part~(1) $I= I_{(H,S)}$ for a saturated hereditary subset $H \subseteq E^0$ and a set of breaking vertices $S \subseteq B_H$, and $E^0 \setminus H$ satisfies (MT1) and (MT2).  Since $I= I_{(H,S)}$ is a primitive ideal, $C^*(E) / I_{(H,S)} \cong C^*(E \setminus (H,S))$ is a primitive $C^*$-algebra.  It follows from Theorem~\ref{primitivitytheorem} that $E \setminus (H,S)$ is downward directed and satisfies the Countable Separation Property.  Because $E \setminus (H,S)$ is downward directed, $E \setminus (H,S)$ contains at most one sink, and thus $S = B_H$ or $S = B_H \setminus \{ v_0 \}$ for some breaking vertex $v_0 \in B_H$.  We consider each possibility.  If $S = B_H$, then $(E \setminus (H,S))^0 = E^0 \setminus H$, and $E^0 \setminus H$ is downward directed and satisfies the Countable Separation Property, so $E^0 \setminus H$ satisfies (MT3) and (MT4).  Thus Proposition~\ref{MT-axioms-tails-prop} implies that $E^0 \setminus H$ is a maximal tail.  On the other hand, if $S = B_H \setminus \{ v_0 \}$ for some breaking vertex $v_0 \in B_H$, then $E \setminus (H,S)$ contains a single sink $v_0'$ coming from the breaking vertex $v_0$.  The fact $E \setminus (H,S)$ is downward directed implies that every vertex in $E \setminus (H,S)$ can reach the sink $v_0'$, and hence every vertex in $E^0 \setminus H$ can reach $v_0$ via a path in $E$.  Thus $E^0 \setminus H = T_{v_0}$.  

For the converse of part~(2), first suppose $I = I_{(H,S)}$ with $E^0 \setminus H$ a maximal tail.  By Proposition~\ref{MT-axioms-tails-prop}, $E^0 \setminus H$ satisfies (MT1), (MT2), (MT3), and (MT4).  The fact that $E^0 \setminus H$ satisfies (MT1) and (MT2) implies that $H$ is a saturated hereditary subset, and it follows from the definition of Condition~(K) that $E \setminus (H,S)$ satisfies Condition~(L).  We now consider the two possibilities described in (2)(a) and (2)(b): If $S=B_H$, then the vertex set of $E \setminus (H,S)$ is $(E \setminus (H,S))^0 = E^0 \setminus H$ and since $E^0 \setminus H$ satisfies (MT3) and (MT4), we see that the graph $E \setminus (H,S)$ is downward directed and satisfies the Countable Separation Property.  If $E^0 \setminus H = T_{v_0}$ for some $v_0 \in B_H$ and $S = B_H \setminus \{ v_0 \}$, then the vertex set of $E \setminus (H,S)$ is $(E \setminus (H,S))^0 = (E^0 \setminus H) \cup \{ v_0' \}$.  Since $v_0 \in B_H$, there is an edge in $E$ from $v_0$ to $E^0 \setminus H$.  Since $E^0 \setminus H = T_{v_0}$ there is a path of positive length in $E$ from $v_0$ to $v_0$.  It follows that there is a path in $E \setminus (H,S)$ from $v_0$ to $v_0'$.  Thus there is a path in $E \setminus (H,S)$ from every vertex in $E \setminus (H,S)$ to $v_0'$, and hence $E \setminus (H,S)$ is downward directed and satisfies the Countable Separation Property.  Thus in either of the cases of (2)(a) or (2)(b), $E \setminus (H,S)$  satisfies Condition~(L), is downward directed, and satisfies the Countable Separation Property.  It follows from Theorem~\ref{primitivitytheorem} that $C^*(E \setminus (H,S)) \cong C^*(E) / I_{(H,S)}$ is a primitive $C^*$-algebra, and hence $I_{(H,S)}$ is a primitive ideal.

For part~(3), first suppose that $I$ is a prime ideal.  From part~(1) $I= I_{(H,S)}$ for a saturated hereditary subset $H \subseteq E^0$ and a set of breaking vertices $S \subseteq B_H$, and $E^0 \setminus H$ satisfies (MT1) and (MT2).  Since $I= I_{(H,S)}$ is a prime ideal, $C^*(E) / I_{(H,S)} \cong C^*(E \setminus (H,S))$ is a prime $C^*$-algebra.  It follows from Proposition~\ref{primeprop}  that $E \setminus (H,S)$ is downward directed.  Because $E \setminus (H,S)$ is downward directed, $E \setminus (H,S)$ contains at most one sink, and thus $S = B_H$ or $S = B_H \setminus \{ v_0 \}$ for some breaking vertex $v_0 \in B_H$.  We consider each possibility.  If $S = B_H$, then $(E \setminus (H,S))^0 = E^0 \setminus H$, and $E^0 \setminus H$ is downward directed, so $E^0 \setminus H$ satisfies (MT3).  Thus Proposition~\ref{MT-axioms-tails-prop} implies that $E^0 \setminus H$ is a cluster of maximal tails.  On the other hand, if $S = B_H \setminus \{ v_0 \}$ for some breaking vertex $v_0 \in B_H$, then the result of part (2)(b) shows that $I$ is primitive ideal.

For the converse of part~(3), first suppose $I = I_{(H,S)}$ with $E^0 \setminus H$ a cluster of maximal tails.  By Proposition~\ref{MT-axioms-tails-prop}, $E^0 \setminus H$ satisfies (MT1), (MT2), and (MT3).  The fact that $E^0 \setminus H$ satisfies (MT1) and (MT2) implies that $H$ is a saturated hereditary subset, and it follows from Remark \ref{KimpliesquotientshaveL} that $E \setminus (H,S)$ satisfies Condition~(L).  We now consider the two possibilities described in (3)(a) and (3)(b): If $S=B_H$, then the vertex set of $E \setminus (H,S)$ is $(E \setminus (H,S))^0 = E^0 \setminus H$ and since $E^0 \setminus H$ satisfies (MT3), we see that the graph $E \setminus (H,S)$ is downward directed.  It follows from Proposition~\ref{primeprop} that $C^*(E \setminus (H,S)) \cong C^*(E) / I_{(H,S)}$ is a prime $C^*$-algebra, and hence $I_{(H,S)}$ is a prime ideal.

If $E^0 \setminus H = T_{v_0}$ for some $v_0 \in B_H$ and $S = B_H \setminus \{ v_0 \}$, then it follows from the result of part (2)(b) that $I$ is a primitive ideal, and hence (see Remark~\ref{primitiveimpliesprime}) that $I$ is a prime ideal.
\end{proof}

\begin{corollary}
If $I_{(H,S)}$ is prime and $S \neq B_H$, then $I_{(H,S)}$ is primitive.
\end{corollary}

The above corollary shows us that if we seek ideals of $C^*(E)$ that are prime but not primitive, we must look for ideals of the form $I_{(H,S)}$ with $S = B_H$.  
In Example \ref{primenotprimitiveexample} below we present examples of primitive ideals arising as in case 2(a) of Theorem \ref{prime-primitive-ideal-description-theorem}, and of   prime but not primitive ideals arising as in case 3(a) of Theorem \ref{prime-primitive-ideal-description-theorem}.  
 On the other hand, we present here  an example of a (necessarily primitive) ideal of a graph C$^*$-algebra which arises as in the identical cases 2(b) and 3(b) of Theorem \ref{prime-primitive-ideal-description-theorem}.  

\begin{example}\label{finitereturnexample}   Let $E$ be the graph  pictured here, where the double arrows with $\infty$ indicate a countable number of edges between the respective vertices.

%\medskip

$$
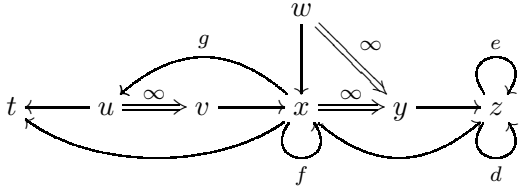
\xymatrix{ & & & w \ar@{=>}[dr]^{\infty}  \ar[d] & & \\
t & u \ar[l] \ar@{=>}[r]^{\infty} & v \ar[r] & x \ar@{=>}[r]^{\infty} \ar@(ul,ur)[ll]_g \ar@(ul,ur)[ll]  \ar@(dl,dr)[lll]
\ar@(dl,dr)_f \ar@(dr,dl)[rr] & y \ar[r] & z 
%\ar@(dr,ur)_e \ar@(dl,ul) 
\ar@(dl,dr)_d \ar@(ul,ur)^e
\\
}
$$

%\bigskip
\bigskip
\noindent This graph is countable, so the maximal tails of $E$ coincide with the clusters of maximal tails of $E$.  In addition, the countability of $E$ implies that $C^*(E)$ is separable, so that the primitive ideals of $C^*(E)$ coincide with the prime ideals of $C^*(E)$.  

There are four maximal tails in $E$, specifically:
$$
T_1 = \{t, u, v, w, x \}, \ T_2 = \{ u, v, w, x \}, \ 
T_3 = \{w \},  \text{ and }  T_4 = \{u,v,w,x,y,z \}.
$$
These maximal tails arise as sets of the form $T_\alpha$ for various (not necessarily unique) elements of $\partial E$.  For instance, $T_1 = T_t$ for the sink $t$, and  $T_3 = T_w$ for the infinite emitter $w$.  As well,   $T_2 = T_u = T_x $ for the infinite emitters $u$ and $x$;  $T_2$ may also be viewed as  $T_{fff\cdots}$ for the infinite path $fff\cdots$, and as $T_\beta$ where $\beta$ is any one of the infinite paths having vertex sequence $xuvxuvxuvx\cdots$.   We see in addition  that $T_4 = T_\alpha$ for any of the (uncountably many different) infinite paths based at $z$ (e.g., $T_4 = T_{eee\cdots} = T_{ddd\cdots} = T_{edede\cdots} = \hdots$).    The complements $H_i = E^0 \setminus T_i$ of these maximal tails give the following saturated hereditary subsets of $E^0$
\begin{align*}
H_1 = \{y, z \}, \quad H_2  = \{ t, y, z \}, \quad 
H_3  = \{t, u, v, x, y, z \}, \ \text{ and } \ H_4   = \{t \} ,
\end{align*}
whose corresponding sets of breaking vertices are easily seen to be 
$$
B_{H_1} =  \{w, x \}, \quad B_{H_2} =  \{ w,x \}, \quad B_{H_3} =  \emptyset, \ \text{ and } \ B_{H_4} = \emptyset.
$$
The only infinite emitter $v_0$ with the property that $v_0$ is a breaking vertex of $E^0 \setminus T_{v_0}$ is the vertex $v_0 = x$.  Thus the primitive (and also prime) ideals of $C^*(E)$ are precisely
$$I_{(H_1, B_{H_1})}, \quad I_{(H_2, B_{H_2})}, \quad I_{(H_2 , B_{H_2} \setminus \{ x \})}, \quad I_{(H_3 , B_{H_3})}, \ \text{ and } \ I_{(H_4, B_{H_4})}. $$
\end{example}

\begin{remark}
For a graph $E$ and field $K$ one may form the $K$-algebra $L_K(E)$, the {\it Leavitt path algebra of } $E$ {\it with coefficients in } $K$.    In \cite{R}, the prime and primitive ideals of $L_K(E)$ are described in terms of the structure of $E$ together with the set of irreducible polynomials of  the $K$-algebra  $K[x,x^{-1}]$.   While the prime and primitive ideal structures of $C^*(E)$ and $L_K(E)$ are quite different in general, there is some intriguing similarity in these structures, as, for example, both Condition (MT3) and the Countable Separation Property  play roles in the results of \cite{R} as well.    
\end{remark}

%--------------------------------------------------------------------------------------------------------------
\section{The prime spectrum and the primitive ideal space of $C^*$-algebras of graphs satisfying Condition~(K)}
%--------------------------------------------------------------------------------------------------------------

The prime spectrum of a commutative ring, consisting of the set of prime ideals of the ring endowed with the Zariski topology, is a valuable tool in both commutative ring theory and algebraic geometry.   In $C^*$-algebra theory, this notion is generalized to noncommutative $C^*$-algebras using the space of primitive ideals endowed with the hull-kernel topology (also called the Jacobson topology).  Any commutative $C^*$-algebra is isomorphic to the space of continuous functions vanishing at infinity on the primitive ideal space of the $C^*$-algebra, and a noncommutative $C^*$-algebra may be viewed as generalized ``noncommutative functions" on the primitive ideal space.  Because of this viewpoint, the study of $C^*$-algebras is sometimes referred to as \emph{noncommutative topology}.  The primitive ideal space is also a useful invariant of a $C^*$-algebra.  Unlike the situation in an arbitrary algebra, any ideal of a $C^*$-algebra is equal to the intersection of all primitive ideals containing it (see \cite[Theorem~5.4.3]{Mur} or \cite[Theorem~A.17(a)]{RW}).  Consequently, the ideal lattice of a $C^*$-algebra may be reconstructed from its primitive ideal space (both the set of primitive ideals and the topology on this set is needed).  Hence the primitive ideal space provides a concise means of recording the entire ideal lattice of a $C^*$-algebra.  This is one of several reasons why the primitive ideal space is of great interest to C*-algebraists.

When a $C^*$-algebra is separable, the prime ideals and the primitive ideals coincide and the prime spectrum is equal to the primitive ideal space.  For nonseparable $C^*$-algebras the primitive ideal space is a (possibly proper) subspace of the prime spectrum.  In this section we compute the prime spectrum and primitive ideal space of the $C^*$-algebra of a graph satisfying Condition~(K), and we describe each of these spaces and their topologies in terms of the graph.

\begin{definition}
If $A$ is a $C^*$-algebra, we define 
$$\spec A := \{ I : \text{$I$ is a prime ideal of $A$} \}.$$
For any subset $X \subseteq \spec A$, we define the closure of $X$ to be $$\overline{X} = \{ I \in \spec A : \bigcap_{J \in X} J \subseteq I \}.$$  This closure operation satisfies the Kuratowski closure axioms (see \cite[\S1.3]{HockingYoung} for more details on the Kuratowski closure axioms), and consequently there is a unique topology on $\spec A$ such that the closed sets of this topology are those $X \subseteq \spec A$ with $\overline{X} = X$.  We call $\spec A$ with this topology the \emph{prime spectrum} of $A$.  

We also define 
$$\prim A := \{ I : \text{$I$ is a primitive ideal of $A$} \}.$$
It follows (see Remark~\ref{primitiveimpliesprime}) that $\prim A \subseteq \spec A$;  we give $\prim A$ the subspace topology it inherits as a subset of $\spec A$.  We call $\prim A$ with this topology the \emph{primitive ideal space} of $A$.
\end{definition}

\begin{remark}
If $A$ is a $C^*$-algebra, the topologies on $\spec A$ and $\prim A$ are always $T_0$, but neither of the topologies is necessarily $T_1$.
\end{remark}

The fact that $\prim A$ is a dense subset of $\spec A$ is straightforward, and likely not new, but since we know of no reference, we provide a proof here.

\begin{proposition}
If $A$ is a $C^*$-algebra, $\prim A$ is a dense subspace of $\spec A$.  Moreover, when $A$ is separable $\prim A = \spec A$.
\end{proposition}

\begin{proof}
Let $A$ be a $C^*$-algebra. Suppose $I \in \spec A$ is a prime ideal.  Every ideal in a $C^*$-algebra is the intersection of the primitive ideals that contain it (see \cite[Theorem~5.4.3]{Mur}).  Thus $I = \bigcap_{J \in \mathcal{S}} J$ where $\mathcal{S} = \{ J \in \prim A : I \subseteq J \}$.  Suppose that $U \subseteq \spec A$ is an open subset with $I \in U$.  Then $X := (\spec A) \setminus U$ is a closed subset of $\spec A$ that does not contain $I$.  Since $X$ is closed, $$X = \overline{X} = \{ K \in \spec A : \bigcap_{L \in X} L \subseteq K \}.$$  Because $I \notin X$, it follows that $\bigcap_{L \in X} L \nsubseteq I$.  Hence there exists an element $a \in \bigcap_{L \in X} L \notin I$.  Since $I = \bigcap_{J \in \mathcal{S}} J$, there exists a primitive ideal $J$ with $a \notin J$ and $I \subseteq J$.  Since $a \notin J$, we have that $\bigcap_{L \in X} L \nsubseteq J$.  Thus $J \notin X$; i.e.,  $J \in U = (\spec A) \setminus X$.  Hence $J \in U \cap \prim A$.  So we have shown that every open set of $\spec A$ containing $I$ intersects $\prim A$ nontrivially, and so the closure of $\prim A$ is $\spec A$.   %and $I$ is a limit point of $\prim A$.  It follows that the closure of $\prim A$ is $\spec A$.  
Finally, if $A$ is separable, then every prime ideal is primitive (by the aforementioned result of Dixmier),  and $\prim A = \spec A$. 
\end{proof}

\begin{definition}
If $E$ is a graph, we let $$\MT (E) := \{ T_\alpha : \alpha \in \partial E \}$$ denote the set of maximal tails of $E$, and let $$\CL (E) := \{ C \subseteq E^0 : \text{$C$ is a cluster of maximal tails} \}$$ denote the set of clusters of maximal tails of $E$.
\end{definition}

\begin{definition}
We say that an infinite emitter $v \in E^0_\textnormal{sing}$ is a \emph{finite-return vertex} if $0 < | \{ e\in s^{-1}(v) : r(e) \geq v \}| < \infty$. We let $$\FR (E) := \{ v \in E^0_\textnormal{sing} : \text{$v$ is a finite-return vertex} \}$$ denote the set of finite-return vertices of $E$.  Note that every element of $\FR (E)$ is an infinite emitter in $E$.  
% subset of the infinite \subseteq E^0_\textnormal{sing}$.
\end{definition}

\begin{example} \label{FR-computed-ex}
We revisit the graph $E$ presented in Example~\ref{finitereturnexample}.  We see that $\FR (E) = \{ x \}$.  The infinite emitter $x$ is a finite-return vertex because there are exactly two edges (namely $f$ and $g$) having source $x$ and whose range can reach $x$.   Moreover, the infinite emitter $w$ is not a finite-return vertex, since none of the edges it emits have ranges that can reach $w$, and the infinite emitter $u$ is not a finite-return vertex, since $u$ emits infinitely many edges from $u$ to $v$ whose ranges can reach $u$.
\end{example}

\begin{remark}\label{finitereturnvertexiffinBsubH}
Let  $v_0 \in E^0_\textnormal{sing}$, and consider  the maximal tail $T_{v_0}$ corresponding to $v_0$.  Then the set $H := E^0 \setminus T_{v_0}$ is saturated hereditary.  By definition, the set of breaking vertices $B_H$ of $H$ is contained in $E^0 \setminus H = E^0 \setminus (E^0 \setminus T_{v_0}) = T_{v_0}$.   By the definition of a breaking vertex, we see immediately that
$$v_0 \in B_H \ \Longleftrightarrow \ v_0 \in \FR(E).$$
\end{remark}

\begin{remark}
Throughout this section we shall consider the disjoint union $$ \CL(E) \sqcup \FR(E).$$  The sets $\CL(E)$ and $\FR(E)$ are disjoint because elements of $\CL(E)$ are subsets of $E^0$, while elements of $\FR(E)$ are elements of $E^0$.  Note, however, that if $v \in E^0_\textnormal{sing}$ with $T_{v} = \{ v \}$ it is possible to have $\{ v \} \in \CL(E)$ and $v \in \FR(E)$.  This occurs in the  graph
$$
\xymatrix{
\bullet^v \ar@{=>}[r]^{\infty}  \ar@(ul,dl) & \bullet \ar@(ur,dr) \\
}  
$$
for instance.
\end{remark}

\smallskip

\begin{definition}
If $X \subseteq \CL(E) \sqcup \FR(E)$, we shall use the notation $X_{T} := X \cap \CL(E)$ and $X_{F} := X \cap \FR(E)$.  Note that $X = X_T \sqcup X_F$.  

For a subset $X \subseteq \CL(E) \sqcup \FR(E)$, we define
$$\mathcal{V}(X) := \bigcup_{C \in X_T} C \cup \bigcup_{v_0 \in X_F} T_{v_0} \subseteq E^0$$
and
\begin{align*}
\overline{X} := \{ C \in \CL(E) &: C \subseteq \mathcal{V}(X) \} \\
&\sqcup \{ v_0 \in \FR(E) : \text{there are infinitely many $e \in E^1$ with} \\
& \qquad \qquad \qquad \qquad \qquad \text{$s(e) =v_0$ and $r(e) \in \mathcal{V}(X)$} \}.
\end{align*}

This closure operation satisfies the Kuratowski closure axioms (see \cite[\S1.3]{HockingYoung} for more details on the Kuratowski closure axioms), and consequently there is a unique topology on $\CL(E) \sqcup \FR(E)$ such that the closed sets of this topology are those subsets $X \subseteq \CL(E) \sqcup \FR(E)$ with $\overline{X} = X$.  In addition, this topology restricts to a topology on the subspace $\MT (E) \sqcup \FR(E)$ where the closed sets are described by the closure operation
\begin{align*}
\overline{X} := \{ C \in \MT(E&) : C \subseteq \mathcal{V}(X) \} \\
&\sqcup \{ v_0 \in \FR(E) : \text{there are infinitely many $e \in E^1$ with} \\
& \qquad \qquad \qquad \qquad \qquad \text{$s(e) =v_0$ and $r(e) \in \mathcal{V}(X)$} \}
\end{align*}
for any $X \subseteq \MT (E) \sqcup \FR(E)$.
\end{definition}

Throughout the proof of the main result of this section (Theorem~\ref{SpecAandPrimAhomeomorphisms}) we will make frequent use of the following result.  Although the statement in \cite{BHRS} is given only for countable graphs, the proof goes through verbatim for arbitrary graphs.

\begin{proposition}[Proposition~3.9 of \cite{BHRS}] \label{BHRSProp}
Suppose $E$ is a graph and $\{ H_j \}_{j \in J}$ is a collection of saturated hereditary subsets of $E$ and that $S_j \subseteq B_{H_j}$ for each $j \in J$.  Let $H := \bigcap_{j \in J} H_j$ and $S := ( \bigcap_{j \in J} H_j \cup S_j) \cap B_H$.  Then $$\bigcap_{j \in J} I_{(H_j,S_j)} = I_{(H,S)}.$$
\end{proposition}

We are now in position to describe, in terms of subsets of the underlying graph $E$,  the prime spectrum and primitive ideal space of $C^*(E)$ when $E$ satisfies Condition (K).

\begin{theorem}\label{SpecAandPrimAhomeomorphisms}
Let $E = (E^0, E^1, r, s)$ be a graph satisfying Condition~(K).  Define a function $h : \CL(E) \sqcup \FR(E) \to \spec A$ as follows:
\begin{center}
for $C \in \CL (E)$, we set $h(C) := I_{(H,B_H)}$ where $H := E^0 \setminus C$, and \\ 
 for $v_0 \in \FR(E)$, we set $h(v_0) := I_{(H, B_H \setminus \{ v_0 \})}$ where $H := E^0 \setminus T_{v_0}$.
\end{center}   Then $h$ is a homeomorphism from $\CL(E) \sqcup \FR(E)$ onto $\spec A$, and $h$ restricts to a homeomorphism from $\MT(E) \sqcup \FR(E)$ onto $\prim A$.
\end{theorem}

\begin{proof}
Throughout the proof, we will find it convenient to use the following notation: For $C \in \CL(E)$, let $H_C := E^0 \setminus C$, and for $v_0 \in \FR(E)$, let $H_{v_0} := E^0 \setminus T_{v_0}$.  Note that for any $X \subseteq \CL(E) \sqcup \FR(E)$, we have by definition that $\mathcal{V}(X) := \bigcup_{C \in X_T} C \cup \bigcup_{v_0 \in X_F} T_{v_0}$, and De Morgan's laws imply $E^0 \setminus \mathcal{V}(X) = \bigcap_{\lambda \in X} H_\lambda$.

It follows from the proof of Theorem~\ref{prime-primitive-ideal-description-theorem}  that $h : \CL(E) \sqcup \FR(E) \to \spec A$ is a bijection, and that $h$ restricts to a bijection from $\MT(E) \sqcup \FR(E)$ onto $\prim A$.  Thus we need only show that $h : \CL(E) \sqcup \FR(E) \to \spec A$ is continuous with continuous inverse.  It suffices to show that $h (\overline{X}) = \overline{h(X)}$ for all $X \subseteq \CL(E) \sqcup \FR(E)$.  To this end, let $X \subseteq \CL(E) \sqcup \FR(E)$.    

If $x \in h (\overline{X})$, then $x = h(\gamma)$ for some $\gamma \in \overline{X}$.  We will consider two cases: $\gamma \in \CL(E)$ and $\gamma \in \FR(E)$.

\smallskip

\noindent \textsc{Case I:} $\gamma = C \in \CL(E)$.  The fact $C \in \overline{X}$ implies $C \subseteq \mathcal{V}(X)$.  Hence $E^0 \setminus  \mathcal{V}(X) \subseteq E^0 \setminus C$, and $\bigcap _{\lambda \in X} H_\lambda \subseteq H_C$.  Proposition~\ref{BHRSProp} implies that $\bigcap_{\lambda \in X} h(\lambda) = I_{(H, S)}$ where $H = \bigcap _{\lambda \in X} H_\lambda$ and $S \subseteq B_H$.  Thus $\bigcap_{\lambda \in X} h(\lambda) \subseteq I_{(H_C, B_{H_C})}$, and $\bigcap_{h(\lambda) \in h(X)} h(\lambda) \subseteq h(C)$, so that $x = h(C) \in \overline{h(X)}$.

\smallskip

\noindent \textsc{Case II:} $\gamma = v_0 \in \FR(E)$.   The fact that $v_0 \in \overline{X}$ implies the set
\begin{equation} \label{inf-set-eq}
\{ e \in E^1 : s(e) = v_0 \text{ and } r(e) \in \mathcal{V}(X) \}
\end{equation}
is infinite.  In particular, $v_0 \geq \mathcal{V}(X)$, and since $\mathcal{V}(X)$ is a union of maximal tails, property~(MT1) implies that $v_0 \in \mathcal{V}(X)$ and $T_{v_0} \subseteq \mathcal{V}(X)$.  Hence $E^0 \setminus \mathcal{V}(X) \subseteq E^0 \setminus T_{v_0}$ and $\bigcap _{\lambda \in X} H_\lambda \subseteq H_{v_0}$.  Proposition~\ref{BHRSProp} implies that $\bigcap_{\lambda \in X} h(\lambda) = I_{(H, S)}$ where $H = \bigcap _{\lambda \in X} H_\lambda$ and $S \subseteq B_H$.  Thus $\bigcap_{\lambda \in X} h(\lambda) \subseteq I_{(H_{v_0}, B_{H_{v_0}})}$.  In addition, since the set in \eqref{inf-set-eq} is infinite, using property (MT1) we see that there are infinitely many edges beginning at $v_0$ and ending in $\mathcal{V}(X)$, and hence there are infinitely many edges beginning at $v_0$ and ending outside of $E^0 \setminus \mathcal{V}(X) =  \bigcap _{\lambda \in X} H_\lambda = H$.  Thus $v_0$ is not a breaking vertex of $H$.  It follows that $\bigcap_{\lambda \in X} h(\lambda) \subseteq I_{(H_{v_0}, B_{H_{v_0}} \setminus \{ v_0 \})}$, and $\bigcap_{h(\lambda) \in h(X)} h(\lambda) \subseteq h(v_0)$, so that $x = h(v_0) \in \overline{h(X)}$.

Combining Case I and Case II above, we have shown $h(\overline{X}) \subseteq \overline{h(X)}$.    For the reverse inclusion, suppose that $x \in \overline{h(X)}$.  By the surjectivity of $h$, we have $x = h(\gamma)$ for some $\gamma \in \CL(E) \sqcup \FR(E)$.  Again, we will consider two cases: $\gamma \in \CL(E)$ and $\gamma \in \FR(E)$.

\smallskip

\noindent \textsc{Case I:} $\gamma = C \in \CL(E)$.  The fact that $h(C) \in \overline{h(X)}$ and the bijectivity of $h$ imply that $\bigcap_{h(\lambda) \in h(X)} h(\lambda) \subseteq h(C)$, and $\bigcap_{\lambda \in X} h(\lambda) \subseteq I_{(H_C, B_{H_C})}$.  Since $h(\lambda)$ has the form $I_{(H_\lambda, S_\lambda)}$ for some $S_\lambda \subseteq B_{H_\lambda}$, it follows from Proposition~\ref{BHRSProp} that $\bigcap_{\lambda \in X} h(\lambda) = I_{(H, S)}$ where $H = \bigcap _{\lambda \in X} H_\lambda$ and $S \subseteq B_H$.  The fact that $\bigcap_{\lambda \in X} h(\lambda) \subseteq I_{(H_C, B_{H_C})}$ then implies that $H \subseteq H_C$, so that $\bigcap _{\lambda \in X} H_\lambda \subseteq H_C$ and $E^0 \setminus \mathcal{V}(X) \subseteq E^0 \setminus C$.  Thus $C \subseteq  \mathcal{V}(X)$ and $C \in \overline{X}$.  Hence $x = h(C) \in h(\overline{X})$.

\smallskip

\noindent \textsc{Case II:} $\gamma = v_0 \in \FR(E)$.   Since $h(v_0) \in \overline{h(X)}$, we have $\bigcap_{h(\lambda) \in h(X)} h(\lambda) \subseteq h(v_0)$, and using the bijectivity of $h$ we obtain $\bigcap_{\lambda \in X} h(\lambda) \subseteq I_{(H_{v_0}, B_{H_{v_0}} \setminus \{ v_0 \})}$.  Proposition~\ref{BHRSProp} implies that $\bigcap_{\lambda \in X} h(\lambda)  = I_{(H,S)}$ where $H = \bigcap _{\lambda \in X} H_\lambda$ and $S \subseteq B_H$.  Since $I_{(H,S)} \subseteq I_{(H_{v_0}, B_{H_{v_0}} \setminus \{ v_0 \})}$, it follows that $v_0 \notin H$ and $v_0 \notin S$.  Let us consider two possibilities: $v_0 \in B_H$ and $v_0 \notin B_H$.

\smallskip

\noindent \textsc{Subcase II}(a): $v_0 \in B_H$.  We show in this subcase that in fact $v_0\in X$.  Since $v_0$ is a breaking vertex of $H$, the set $\{ e \in E^1 : s(e) = v_0 \text{ and } r(e) \notin H \}$ is finite.  In addition, since $v_0 \notin S$, it follows $p_{v_0} - \sum_{s(e)= v_0, r(e) \notin H} s_es_e^* \notin I_{(H,S)} = \bigcap_{\lambda \in X} h(\lambda)$.  Thus there exists 
$\delta \in X$ such that $p_{v_0} - \sum_{s(e)= v_0, r(e) \notin H} s_es_e^* \notin h(\delta)$.  In particular, this implies that $p_{v_0} \notin h(\delta)$ and $v_0 \notin H_\delta$.  Since $H \subseteq H_\delta$, we see that $$\{ e \in E^1 : s(e) = v_0 \text{ and } r(e) \notin H_\delta \} \subseteq \{ e \in E^1 : s(e) = v_0 \text{ and } r(e) \notin H \},$$ and hence $\{ e \in E^1 : s(e) = v_0 \text{ and } r(e) \notin H_\delta \}$ is finite.  Because $v_0$ is a finite-return vertex, $v_0 \notin H_\delta$, and $H_\delta$ is hereditary, there exists at least one edge from $v_0$ to $E^0 \setminus H_\delta$.  Thus $0 < |\{ e \in E^1 : s(e) = v_0 \text{ and } r(e) \notin H_\delta \}| < \infty$ and $v_0$ is a breaking vertex of $H_\delta$; i.e.,   $v_0 \in B_{H_\delta}$.  In addition, if $e \in E^1$ with $r(e) \in H_\delta$, then $p_{r(e)} \in I_{(H_\delta, \emptyset)} \subseteq h(\delta)$, and $s_es_e^* = s_ep_{r(e)} s_e^* \in h(\delta)$.  Since $\{ e \in E^1 : s(e) = v_0 \text{ and } r(e) \in H_\delta \setminus H \}$ is a finite set, $$\sum_{s(e) = v_0, r(e) \notin H_\delta \setminus H} s_es_e^* \in h(\delta).$$  Because $p_{v_0} - \sum_{s(e)= v_0, r(e) \notin H} s_es_e^* \notin h(\delta)$ and $H \subseteq H_\delta$, it follows that

\begin{align*}
p_{v_0} - &\sum_{s(e) =v_0, r(e) \notin H_\delta} s_es_e^*   =  p_{v_0} -  \left( \sum_{s(e)= v_0, r(e) \notin H} s_es_e^* + \sum_{s(e) = v_0, r(e) \notin H_\delta \setminus H} s_es_e^* \right)\\
& \qquad \qquad =  \left(p_{v_0} -  \sum_{s(e)v_0, r(e) \notin H_\delta} s_es_e^* \right) - \sum_{s(e) = v_0, r(e) \notin H_\delta \setminus H} s_es_e^* \notin h(\delta) \\
\end{align*}
Thus the gap projection $p_{v_0} - \sum_{s(e)=v_0, r(e) \notin H_\delta} s_es_e^*$ is not in $h(\delta)$, and hence by the way $h$ is defined $h(\delta)$ must be equal to $I_{(H_{v_0}, B_{H_{v_0}} \setminus \{ v_0 \})}$, and $\delta = v_0$.  Thus $v_0 = \delta \in X \subseteq \overline{X}$ and $x = h(v_0) \in h(\overline{X})$.

\smallskip

\noindent \textsc{Subcase II}(b):  $v_0 \notin B_H$.  Then $v_0$ is not a breaking vertex of $H$.  Since $v_0 \notin H$, the set $H$ is hereditary, and $v_0$ is a finite-return vertex, there is at least one edge from $v_0$ to $E^0 \setminus H$.  Because $v_0$ is not a breaking vertex of $H$, it then follows that $$\{ e \in E^1 : s(e) = v_0 \text{ and } r(e) \notin H \}$$ is an infinite set.  But $r(e) \notin H = \bigcap_{\lambda \in X} H_\lambda$ implies $r(e) \in E^0 \setminus \bigcap_{\lambda \in X} H_\lambda = \bigcup_{\lambda \in X} E^0 \setminus H_\lambda = \mathcal{V}(X)$.  Thus $$\{ e \in E^1 : s(e) = v_0 \text{ and } r(e) \in \mathcal{V}(X) \}$$ is an infinite set. By definition, $v_0 \in \overline{X}$, so that $x = h(v_0) \in h(\overline{X})$.

\smallskip

Combining Case I and Case II above, we have that $ \overline{h(X)} \subseteq h(\overline{X})$.  Hence  $h (\overline{X}) = \overline{h(X)}$ for all $X \subseteq \CL(E) \sqcup \FR(E)$, and $h$ is a homeomorphism.
\end{proof}

When $E$ has no finite-return vertices (which occurs, for example, if $E$ is either a row-finite graph or a graph with no cycles), then the prime spectrum and primitive ideal space of $C^*(E)$ take a much simpler form.  Since there are no finite-return vertices, the prime spectrum is homeomorphic to the space $\CL(E)$ of clusters of maximal tails, and the primitive ideal space is homeomorphic to the space $\MT (E)$ of maximal tails.  Moreover, the closure operation in the topology is easier to describe.  We give the details of this in the following corollary.

\begin{corollary} \label{RF-MaxTail-cor}
Let $E = (E^0, E^1, r, s)$ be a graph satisfying Condition~(K) and with no finite-return vertices. (This occurs, for example, if $E$ is a row-finite graph satisfying Condition~(K) or if $E$ is a graph with no cycles.)  Give $\CL(E)$ the topology whose closed sets are determined by the closure operation
$$\overline{X} := \{ C \in \CL(E): C \subseteq \bigcup_{D \in X} D \}$$ for $X \subseteq \CL(E)$.  Also give $\MT(E)$ the subspace topology it inherits as a subset of $\CL (E)$.  Define a function $h : \CL(E) \to \spec A$ by
$$h(C) := I_{(H_C,B_{H_C})}$$ where $H_C := E^0 \setminus C$.
Then $h$ is a homeomorphism from $\CL(E)$ onto $\spec A$, and $h$ restricts to a homeomorphism from $\MT(E)$ onto $\prim A$.
\end{corollary}

\begin{remark}
In \cite[Theorem~6.3]{BPRS} the authors describe the primitive ideal space of the $C^*$-algebra of a row-finite countable graph satisfying Condition~(K), and their result is a special case of Corollary~\ref{RF-MaxTail-cor}.  Although the authors of \cite{BPRS} describe the closure operation in \cite[Theorem~6.3]{BPRS} as $\overline{X} := \{ C \in \CL(E): C \geq \bigcup_{D \in X} D \}$, we point out that property~(MT1) of unions of maximal tails implies that $C \geq \bigcup_{D \in X} D$ if and only if $C \subseteq \bigcup_{D \in X} D$.
\end{remark}

\begin{example}
We once again revisit the graph $E$ presented in Example~\ref{finitereturnexample} and consider the topology on $\prim C^*(E)$.  In Example~\ref{finitereturnexample} we computed the maximal tails of $E$ and primitive ideals of $C^*(E)$, and in Example~\ref{FR-computed-ex} we showed that $\FR(E) = \{ x \}$.  Recall that $H_4 = \{t\}$, and that the maximal tail $T_4 = \{ u,v,w,x,y,z\} = E^0 \setminus \{t\}$ may be viewed as $T_\alpha$ for any infinite path $\alpha$ based at $z$.  For clarity, we choose $\alpha = eee\cdots$.   Let $$\{ I_{(H_{4}, B_{H_{4}})} \}$$ be the singleton subset of $\prim C^*(E)$ containing only the primitive ideal $I_{(H_{4}, B_{H_{4}})}$.  This subset corresponds to the singleton set $X = \{ T_{eee\cdots} \}$ in $\MT (E) \sqcup \FR(E)$, and we see that  $\mathcal{V} (X) = \{ u, v, w, x, y, z \}$ and $\overline{X} = \{T_2, T_3, T_{4} \} \sqcup \{ x \}$.  Thus the closure of the singleton set $\{ I_{(H_{4}, B_{H_{4}})} \}$ in $\prim C^*(E)$ is 
$$\overline{ \{ I_{(H_{4}, B_{H_{4}})} \} } = \{ I_{(H_2, B_{H_2})}, \  I_{(H_3, B_{H_3})}, \  I_{(H_{4}, B_{H_{4}})}, 
\ I_{(H_2, B_{H_2} \setminus \{ x \})} \},$$ 
which contains exactly four points of $\prim C^*(E)$.  In particular, since $\prim C^*(E)$ contains singleton sets that are not closed,  $\prim C^*(E)$ is not Hausdorff.
\end{example}

\begin{example}\label{primenotprimitiveexample}
Let $X$ be any uncountable set and let $E_A(X)$ be the graph from Example~\ref{EsubTsubXexample}.  We shall use Theorem~\ref{SpecAandPrimAhomeomorphisms} to give a description of $\spec C^*(E_A(X))$ and $\prim C^*(E_A(X))$ in terms of $X$.

Since $X$ is uncountable, $E_A(X)$ has an uncountable number of vertices, and every vertex in $E_A(X)$ is an infinite emitter. We also observe that $E_A(X)$ has no cycles, and therefore $E_A(X)$ satisfies Condition~(K) and $E_A(X)$ has no finite-return vertices.  It follows from Theorem~\ref{SpecAandPrimAhomeomorphisms} (and Corollary~\ref{RF-MaxTail-cor}) that $\spec C^*(E_A(X))$ is homeomorphic to $\CL(E)$, and that $\prim C^*(E_A(X))$ is homeomorphic to $\MT(E)$.  

Let $P(X)$ denote the collection of nonempty subsets of $X$, and let $P_\textnormal{co} (X)$ denote the collection of nonempty countable subsets of $X$.  For each $S \in P(X)$ define $$\gamma_S := \{ v_A : \text{$A$ is a nonempty finite subset of $S$} \}.$$  It is easy to show that $\gamma_S$ satisfies properties (MT1) and (MT3), and since every vertex in $E_A(X)$ is an infinite emitter, $\gamma_S$ trivially satisfies (MT2).  In addition, $\gamma_S$ satisfies (MT4) if and only if $S$ is countable.  Thus $\gamma_S$ is a cluster of maximal tails for any $S \in P(X)$, and $\gamma_S$ is a maximal tail if and only if $S \in P_\textnormal{co} (X)$.

Moreover, we see that the map $S \mapsto \gamma_S$ is a bijection from $P(X)$ onto $\CL (E_A(X))$ with inverse given by $\gamma \mapsto \bigcup_{v_A \in \gamma} A$, and that this bijection restricts to a bijection from $P_\textnormal{co} (X)$ onto $\MT(E_A(X))$.

Give $P(X)$ the topology whose closed sets are determined by the following closure operation: For $\mathcal{S} \subseteq P(X)$ define
\begin{align*}
\overline{ \mathcal{S} } := \{ T \in P(X) : \text{for} & \text{ each finite subset $A \subseteq T$ there} \\
 & \text{ exists $S \in \mathcal{S}$ with $A \subseteq S$} \}.
\end{align*}
It is straightforward to verify that this closure operation satisfies the Kuratowski closure axioms, 
and thus defines a unique topology on $P(X)$ for which the closed sets are those subsets $\mathcal{S} \subseteq P(X)$ with $\overline{\mathcal{S}} = \mathcal{S}$.  Let us define $\phi : P(X) \to \CL (E)$ by $\phi(S) := \gamma_S$.  We have shown above that $\phi$ is a bijection.  We shall now show that $\phi$ is also a homeomorphism.  To do this it will be convenient to apply the following claim, for which we provide a proof.

\smallskip

\noindent \textbf{Claim:} $\overline{ \{ \gamma_S : S \in \mathcal{S} \} } = \{ \gamma_S : S \in \overline{\mathcal{S}} \}$.

\smallskip

{\it Proof.}  Let $\gamma_T \in \overline{ \{ \gamma_S : S \in \mathcal{S} \} }$.  Then $\gamma_T \subseteq \bigcup_{S \in \mathcal{S}} \gamma_S$.  For each nonempty finite set $A \subseteq T$, we have $v_A \in \gamma_T$ and hence there exists $S \in \mathcal{S}$ with $v_A \in \gamma_S$, so that $A \subseteq S$.  If follows from the definition of $\overline{S}$ that $T \in \overline{S}$, and hence $\gamma_T$ is in the set on the right-hand-side of the above equation.

For the reverse inclusion, let $\gamma_T$ be in the right-hand-side of the above equation; that is, consider $\gamma_T$ for $T \in \overline{S}$.  Then for any $v_A \in \gamma_T$, it follows that $A$ is a finite subset of $T$, and since $T \in \overline{\mathcal{S}}$ there exists $S \in \mathcal{S}$ such that $A \subseteq S$.  Hence $v_A \in \gamma_S$.  Thus $\gamma_T \subseteq \bigcup_{S \in \mathcal{S}} \gamma_S$, and $\gamma_T$ is in the left-hand-side of the above equation, thereby establishing the Claim. 
\smallskip

Using the Claim, we see that for any subset $\mathcal{S} \subseteq P(X)$, we have
$$\overline{\phi(\mathcal{S})} = \overline{\{ \gamma_S : S \in \mathcal{S} \}} = \{ \gamma_S : S \in \overline{\mathcal{S}} \} = \phi (\overline{\mathcal{S}})$$
Thus $\phi : P(X) \to \CL (E)$ is a homeomorphism, and $\phi$ restricts to a homeomorphism $\phi|_{P_\textnormal{co}} : P_\textnormal{co} (X) \to \MT (E)$.  Hence $\spec C^*(E_A(X))$ is homeomorphic to $P(X)$ with the topology defined above, and $\prim C^*(E_A(X))$ is homeomorphic to $P_\textnormal{co} (X)$ with the topology defined above.
\end{example}

\end{document}